\documentclass[11pt,reqno]{amsart}
\usepackage[T1]{fontenc}
\usepackage{lmodern,microtype}
\usepackage[margin=1.05in]{geometry}
\usepackage{amsmath,amssymb,amsthm,mathtools,booktabs,array,graphicx}
\usepackage{xcolor,listings}
\usepackage[hidelinks]{hyperref}
\hypersetup{pdftitle={An explicit rational lemniscate counterexample with certified separators},pdfauthor={Jiang Yang and Xin Zhang}}
\newtheorem{theorem}{Theorem}[section]
\newtheorem{lemma}[theorem]{Lemma}
\newtheorem{proposition}[theorem]{Proposition}
\newtheorem{corollary}[theorem]{Corollary}
\theoremstyle{remark}\newtheorem{remark}[theorem]{Remark}
\newcommand{\C}{\mathbb C}
\newcommand{\R}{\mathbb R}
\newcommand{\Q}{\mathbb Q}
\newcommand{\HH}{\mathcal H^1}
\newcommand{\eps}{\varepsilon}
\DeclareMathOperator{\Rea}{Re}

\title[A rational lemniscate counterexample]{An explicit rational lemniscate counterexample with certified separators}
\author[Jiang Yang]{Jiang Yang}
\address[Jiang Yang]{School of Mathematical Sciences, Guangxi Minzu University, Nanning 530006, P.R. China}
\email{yangjiangdy@126.com}
\author[Xin Zhang]{Xin Zhang$^{*}$}
\address[Xin Zhang]{School of Mathematical Sciences, Guangxi Minzu University, Nanning 530006, P.R. China}
\email{15247161004@163.com}
\thanks{$^{*}$Corresponding author: Xin Zhang (\href{mailto:15247161004@163.com}{15247161004@163.com}).}
\date{23 September 2026}
\begin{document}
\begin{abstract}
In the setting of the previously reported degree-seven counterexample of ani, we present an explicit monic polynomial with real rational coefficients and a direct proof by certified separators. The polynomial has seven simple zeros in the open unit disk. Every continuous path joining distinct zeros in its closed unit sublevel set has image of one-dimensional Hausdorff measure greater than $2+10^{-8}$. Five piecewise-linear graphs exclude twenty of the twenty-one zero pairs; two short forbidden segments force a quantitative detour for the remaining pair. Exact rational Bernstein certificates verify the graph inequalities on forty-four terminal intervals and their unbounded tails. The proof does not require critical-point counting or analytic inverse branches. We separately record the general continuum and dilation arguments that transfer path-length obstructions to Hausdorff measure and to closed sublevel sets. Those transfers are not asserted as new principles. The focus is the explicit real-coefficient construction, its separation certificates, and its verified numerical margin, rather than a first negative answer to the length-two question or a disproof of every larger universal bound.
\end{abstract}
\maketitle

\section{Introduction and statement of the result}
In Problem 5 of \emph{Metric properties of polynomials}, Erd\H{o}s, Herzog and Piranian~\cite[p.~139]{EHP} asked whether a monic polynomial of degree at least two, all of whose zeros lie in the open unit disk, must have two zeros joined by a path of length less than $2$ in the open sublevel set $\{z:|P(z)|<1\}$. We call this the \emph{length-two question}. It is also indexed as Erd\H{o}s problem 1041. Multiplicities cause no difficulty for a counterexample having only simple zeros.

Cook's preprint~\cite{Cook} reports a degree-seven counterexample due to the contributor ani and gives a Lean formalization. Its principal positive results concern monic trinomials and a sufficient condition in terms of the least critical-value modulus; these are compatible with the reported counterexample. The pinned counterexample source~\cite{AniFormal} specifies a Gaussian-rational instance and proves that every root-joining path in the open unit sublevel set has total variation greater than $2$. This prior negative result is the starting context of the present paper, not a conclusion for which priority is claimed here.

The present work supplies a degree-seven polynomial with real rational coefficients, together with a proof based on five global graph separators and exact rational certificates. A cubic factorization locates the zeros; a sign table excludes all but one pair; and two local obstacles give the explicit length excess $10^{-8}$. The proof of these facts does not use critical-point enumeration or analytic continuation of inverse branches. Both this construction and the pinned instance of ani share the general device of a structured perturbation of $z^7-1$ followed by contraction of the zero circle. Their displayed coefficients and verification data differ. We do not assert that the perturbation mechanism itself is new, or claim novelty solely from the difference of coefficients.

The strongest formulation is stated in Theorem~\ref{thm:main} for completeness. Its use of Hausdorff measure and a closed sublevel set must be distinguished from the construction-specific work. Lemma~\ref{lem:continuum} converts a lower bound for all rectifiable connectors into a lower bound for the Hausdorff measure of any continuous connector's image. Proposition~\ref{prop:transfers} also shows how a slight dilation converts any suitable open-sublevel counterexample into a closed-sublevel counterexample with a positive margin. These are general consequences, not separate claims of originality. What is verified specifically here is the stated real rational polynomial, the complete separator data, and the numerical margin for that fixed instance.

For a continuous path $\gamma:[0,1]\to\C$, write $\ell(\gamma)$ for its total variation, allowing $+\infty$, and write $\HH(\gamma([0,1]))$ for the one-dimensional Hausdorff measure of its image. We use the normalization in which a straight line segment has measure equal to its Euclidean length. These two notions of length can differ for a path that retraces itself. Both are covered below.

Set
\begin{equation}\label{eq:constants}
 \eps=\frac1{1000},\qquad \delta=10^{-30},\qquad r=1-\delta,
 \qquad \eta=10^{-8},\qquad a=\frac{11}{4},\quad b=-\frac{55}{16},
\end{equation}
and define
\begin{align}
 Q(z)&=z^7-1+\eps^6(z-z^6)+b\eps^5(z^2-z^5)
                +a\eps^4(z^3-z^4),\label{eq:Q}\\
 P(z)&=r^7Q(z/r).\label{eq:P}
\end{align}
The scaling in \eqref{eq:P} preserves monicity.

\begin{theorem}\label{thm:main}
The polynomial $P\in\Q[z]$ defined by \eqref{eq:constants}--\eqref{eq:P} is monic of degree seven and has seven simple zeros, each of modulus $r<1$. If a continuous path $\gamma:[0,1]\to\C$ joins two distinct zeros of $P$ and satisfies
\[
 |P(\gamma(t))|\leq1\qquad(0\leq t\leq1),
\]
then
\begin{equation}\label{eq:main}
 \HH(\gamma([0,1]))>2+\eta.
\end{equation}
In particular, $\ell(\gamma)>2+\eta$.
\end{theorem}

\begin{corollary}
The polynomial $P$ is a counterexample to the length-two question, whether length is interpreted as total variation or as Hausdorff measure of the path image.
\end{corollary}

\begin{remark}\label{rem:scope}
The theorem concerns the fixed threshold $2$. It does not disprove the existence of an absolute constant $C>2$ with the corresponding universal path property, nor does it establish that degree seven is minimal. The label JSP-000864 motivates this investigation, but no assertion about the interpretation, eligibility rules, or adjudication of a prize problem is needed for the mathematical statement above.
\end{remark}

The proof proceeds in four steps. We locate all zeros by a cubic identity. Five certified graphs then separate all pairs except one conjugate pair. A local detour estimate rules out short rectifiable paths for that pair. Finally, a finite-length continuum lemma proves \eqref{eq:main} for arbitrary continuous paths. The exact computations are specified in Section~\ref{sec:certificates} and reproduced by the checker in Appendix~\ref{app:checker}.

\section{Exact location of the zeros}\label{sec:roots}
Put
\[
 A=1-\eps^6,\qquad B=A-b\eps^5,\qquad C=B-a\eps^4,
 \qquad T_0(u)=8u^3+4Au^2+2(B-3)u+C-2A.
\]
Direct expansion gives the identity
\begin{equation}\label{eq:cubic}
 Q(z)=(z-1)z^3T_0\!\left(\frac{z+z^{-1}}2\right).
\end{equation}
The apparent singularity at zero is removable: the second factor has ascending coefficients $(1,A,B,C,B,A,1)$.

Exact rational substitution yields the signs in Table~\ref{tab:roots}. The three intervals are disjoint. The intermediate value theorem therefore provides three distinct real roots $u_1,u_2,u_3$ of $T_0$, in the listed order. A cubic has no further roots, and these three roots are simple. Since each lies in $(-1,1)$, the zeros of $Q$ are precisely
\begin{equation}\label{eq:alpha}
 1,\quad \alpha_j=u_j+iv_j,\quad \overline{\alpha_j},
 \qquad v_j=\sqrt{1-u_j^2}>0,\quad 1\leq j\leq3.
\end{equation}
They are seven distinct points on the unit circle. Thus the zeros of $P$ are simple and lie on $|z|=r$.

\begin{table}[ht]
\caption{Rational root enclosures and endpoint signs of $T_0$.}\label{tab:roots}
\centering
\begin{tabular}{cccc}
\toprule
$j$ & interval for $u_j$ & endpoint signs & interval for $v_j$\\
\midrule
1 & $(3/5,13/20)$ & $-,+$ & $(3/4,4/5)$\\
2 & $(-1/4,-1/5)$ & $+,-$ & $(24/25,49/50)$\\
3 & $(-91/100,-9/10)$ & $-,+$ & $(41/100,11/25)$\\
\bottomrule
\end{tabular}
\end{table}

The enclosures for $v_j$ follow by squaring the rational endpoints. For example, if $(l,u)$ encloses $u_j$ and $(v_-,v_+)$ is the proposed enclosure for $v_j$, it suffices that
\[
 \max(|l|,|u|)^2+v_-^2<1,
 \qquad \min(|l|,|u|)^2+v_+^2\geq1.
\]
These inequalities hold for each row. In particular, both upper-left zeros satisfy
\begin{equation}\label{eq:leftbox}
 -\frac{91}{100}<u<-\frac15,\qquad v>\frac25,
 \qquad u^2+v^2=1.
\end{equation}

\section{Five global separators}\label{sec:separators}
\subsection{The scaled polynomial}
Define
\begin{equation}\label{eq:g}
 g(w)=w^7-\eps^5w^6-b\eps^3w^5-a\eps w^4+aw^3+bw^2+w.
\end{equation}
Then
\begin{equation}\label{eq:scaling}
 Q(\eps w)=-1+\eps^7g(w),\qquad
 P(r\eps w)=r^7\bigl(-1+\eps^7g(w)\bigr).
\end{equation}
Let $\kappa=1/2000$. The exact rational inequality
\begin{equation}\label{eq:margin}
 r^7(1+\eps^7\kappa)>1
\end{equation}
shows that
\begin{equation}\label{eq:forbidden}
 \Rea g(w)<-\kappa\quad\Longrightarrow\quad |P(r\eps w)|>1.
\end{equation}
Indeed, the modulus of $-1+\eps^7g(w)$ is at least the absolute value of its real part.

\subsection{Explicit graphs}
Table~\ref{tab:vertices} specifies three polygonal curves in the $w=x+iy$ plane. Every terminating decimal in this paper denotes its exact rational value. Join consecutive vertices by line segments and attach the two rays $\lambda p_0$ and $\lambda p_m$, $\lambda\geq1$, at the first and last vertices.

\begin{table}[ht]
\caption{Base graphs. Vertices are in the original $(x,y)$ coordinates.}\label{tab:vertices}
\centering\small
\begin{tabular}{clp{9.1cm}}
\toprule
Graph & $(T,S)$ & Ordered vertices $(x,y)$\\
\midrule
$A_0$ & $(y,x)$ & $(4,-2),\ (.71,-.1),\ (.645,-.03),$\newline $(.645,.035),\ (.88,.25),\ (4,2)$\\[4pt]
$B_+$ & $(x-y,x+y)$ & $(1,4),\ (.53,.84),\ (.5325,.8325),$\newline $(.565,.81),\ (.74,.81),\ (1.1625,.9525),\ (4,2)$\\[4pt]
$C_+$ & $(x+y,-x+y)$ & $(-4,0),\ (-1.1,.555),\ (-.925,.73),$\newline $(-.925,.76),\ (-.93,.77),\ (-3,4)$\\
\bottomrule
\end{tabular}
\end{table}

In each row the linear coordinate change is invertible, the vertex $T$-coordinates strictly increase, and the first and last $T$-coordinates have opposite signs. Consequently the whole unbounded curve is the graph $S=f(T)$ of a continuous piecewise-linear function on $\R$. Define $B_-$ and $C_-$ by complex conjugation of $B_+$ and $C_+$, respectively.

\begin{lemma}\label{lem:graphcert}
Every point of each of the five complete graphs $A_0,B_+,B_-,C_+,C_-$ satisfies $\Rea g(w)<-\kappa$.
\end{lemma}

\begin{proof}
For a segment with endpoints $p,q$, express
\[
 R(t)=\Rea g(p+t(q-p))+\kappa=\sum_{k=0}^7 a_kt^k,
 \qquad 0\leq t\leq1,
\]
in the Bernstein basis of degree seven:
\begin{equation}\label{eq:bernstein}
 R(t)=\sum_{j=0}^7\beta_j\binom7j t^j(1-t)^{7-j},
 \qquad \beta_j=\sum_{k=0}^j a_k\frac{\binom jk}{\binom7k}.
\end{equation}
The basis functions are nonnegative and sum to one. Thus negative Bernstein coefficients certify strict negativity on the whole segment. If necessary, subdivide the parameter interval at its midpoint. The Bernstein coefficients on the two halves are obtained by repeatedly averaging adjacent coefficients; the left edge and reversed right edge of this de Casteljau triangle give the two new coefficient lists.

For the vertices in Table~\ref{tab:vertices}, rational arithmetic gives the terminal interval counts in Table~\ref{tab:cells}. All terminal Bernstein coefficients are strictly negative. The maximum subdivision depth is six. Appendix~\ref{app:checker} gives the entire checking program; the supplement records all forty-four parameter intervals and their negative upper bounds. In particular, the calculation certifies entire segments rather than sampled points.

For an endpoint ray write
\[
 \Rea g(\lambda p)=\sum_{k=0}^7c_k\lambda^k.
\]
For each of the six base rays the rational calculation gives
\begin{equation}\label{eq:ray}
 U=c_7+\sum_{k=0}^6\max(c_k,0)<-\kappa.
\end{equation}
For $\lambda\geq1$, discard the nonpositive lower-degree terms and bound each positive $\lambda^k$ by $\lambda^7$. This yields $\Rea g(\lambda p)\leq U\lambda^7<-\kappa$. The base graphs are therefore certified including both infinite tails. The conjugate graphs satisfy the same inequality because $g$ has real coefficients.
\end{proof}

\begin{table}[ht]
\caption{Numbers of terminal Bernstein intervals for successive base segments.}\label{tab:cells}
\centering
\begin{tabular}{ccc}
\toprule
Graph & terminal counts in vertex order & total\\
\midrule
$A_0$ & $6,1,2,3,1$ & 13\\
$B_+$ & $7,1,2,2,5,1$ & 18\\
$C_+$ & $5,1,1,1,5$ & 13\\
\bottomrule
\end{tabular}
\end{table}

\subsection{The signs at the zeros}
For each graph $S=f(T)$ define the continuous function
\begin{equation}\label{eq:J}
 J(z)=S\!\left(\frac{z}{r\eps}\right)
       -f\!\left(T\!\left(\frac{z}{r\eps}\right)\right).
\end{equation}
By Lemma~\ref{lem:graphcert} and \eqref{eq:forbidden}, its zero set lies in $\{|P|>1\}$. Hence $J$ has constant nonzero sign along every path in $\{|P|\leq1\}$.

Rational interval evaluation on the rectangles in Table~\ref{tab:roots} gives Table~\ref{tab:signs}. Here is a precise way to perform the evaluation without approximating the zeros. At $r\alpha_j$, the input to the graph function is $\alpha_j/\eps$. For each graph, the $T$-image of the entire scaled rectangle belongs to one affine branch $f(T)=mT+d$. The expression $S-mT-d$ is then affine on the rectangle. Its extremal bounds are obtained by choosing a lower or upper endpoint in each coordinate according to the sign of its coefficient. The resulting bounds are strictly of the indicated signs. This is also checked in Appendix~\ref{app:checker}.

\begin{table}[ht]
\caption{Signs of the five separating functions at the zeros of $P$.}\label{tab:signs}
\centering
\begin{tabular}{cccccc}
\toprule
Zero & $A_0$ & $B_+$ & $B_-$ & $C_+$ & $C_-$\\
\midrule
$r$ & $+$ & $-$ & $-$ & $-$ & $-$\\
$r\alpha_1$ & $-$ & $+$ & $-$ & $-$ & $-$\\
$r\overline{\alpha_1}$ & $-$ & $-$ & $+$ & $-$ & $-$\\
$r\alpha_2$ & $-$ & $-$ & $-$ & $-$ & $-$\\
$r\overline{\alpha_2}$ & $-$ & $-$ & $-$ & $-$ & $-$\\
$r\alpha_3$ & $-$ & $-$ & $-$ & $+$ & $-$\\
$r\overline{\alpha_3}$ & $-$ & $-$ & $-$ & $-$ & $+$\\
\bottomrule
\end{tabular}
\end{table}

\begin{proposition}\label{prop:pair}
If two distinct zeros of $P$ can be joined by a continuous path in $\{|P|\leq1\}$, they are $r\alpha_2$ and $r\overline{\alpha_2}$.
\end{proposition}
\begin{proof}
Endpoints of an allowed path must have equal signs for all five functions \eqref{eq:J}, by the intermediate value theorem. Only the two indicated rows of Table~\ref{tab:signs} agree.
\end{proof}
No assertion that the remaining pair can actually be joined is required for the counterexample.

\section{A quantitative obstruction for the remaining pair}\label{sec:length}
\subsection{Forbidden points in the \texorpdfstring{$Q$}{Q}-plane}
Dividing an allowed path for $P$ by $r$ gives a path in $\{|Q|\leq r^{-7}\}$. For $t>0$, every nonconstant term of $Q(-t)$ is negative. Hence
\[
 |Q(-t)|\geq1+\eps^6t.
\]
The rational inequality $r^{-7}-1<8\delta$ shows that any allowed negative real point $x$ satisfies
\begin{equation}\label{eq:D}
 |x|<D:=8\delta/\eps^6=8\cdot10^{-12}.
\end{equation}

Let $h=\eps/100=10^{-5}$ and set
\begin{equation}\label{eq:shortsegments}
 S_\pm=\{x\pm ih:-3h\leq x\leq-h/10\}.
\end{equation}
\begin{lemma}\label{lem:short}
For every $z\in S_+\cup S_-$, one has $|Q(z)|>r^{-7}$.
\end{lemma}
\begin{proof}
Write $z=\eps(X+i/100)$ on $S_+$, so that $-3/100\leq X\leq-1/1000$. Put $W=1/25$; then $|X+i/100|\leq W$. Expanding the low-degree terms of \eqref{eq:g} and bounding the remaining terms by their moduli gives
\begin{align}\label{eq:horizontalbound}
 \Rea g(X+i/100)\leq{}&W^7+\eps^5W^6+(-b)\eps^3W^5+a\eps W^4\notag\\
 &-\frac1{1000}\left(1-\frac{3a}{10000}\right)
       +\frac{-b}{10000}<-\frac1{2000}.
\end{align}
Indeed, $a\Rea(X+i/100)^3=aX^3-3aX/10000$ and $b\Rea(X+i/100)^2=bX^2-b/10000$. The terms $aX^3$ and $bX^2$ are nonpositive, and $1-3a/10000>0$. The last comparison in \eqref{eq:horizontalbound} is a rational inequality. Equivalently, one Bernstein coefficient list certifies the whole segment in the supplied checker. Equations~\eqref{eq:scaling} and \eqref{eq:margin} prove the required bound. Conjugation gives the assertion for $S_-$.
\end{proof}

\subsection{Localization of a real crossing}
Suppose that an allowed rectifiable path $\gamma$ joins $r\alpha_2$ to $r\overline{\alpha_2}$ and has $\ell(\gamma)\leq2+\eta$. Write $\alpha=\alpha_2=u+iv$ and $\Gamma=\gamma/r$. Choose a parameter at which $\Gamma$ is real and denote its value by $x_0$. The two portions of $\Gamma$ have total length at least $2|\alpha-x_0|$. Consequently
\begin{equation}\label{eq:budget}
 |\alpha-x_0|\leq M:=\frac{1+\eta/2}{r}.
\end{equation}
If $x_0\geq0$, then \eqref{eq:leftbox} implies
\[
 1+x_0^2-2ux_0\leq M^2,
 \qquad x_0\leq\frac52(M^2-1)<3\eta.
\]
If $x_0<0$, equation~\eqref{eq:D} gives the same bound. Thus
\begin{equation}\label{eq:crossing}
 |x_0|<3\eta.
\end{equation}
The strict rational comparisons used here are verified in the checker.

\subsection{The forced detour}
Before reaching $x_0$, the path crosses the line $\Im z=h$ at some point $q=x+ih$. Lemma~\ref{lem:short} gives $x<-3h$ or $x>-h/10$. Put $k=x/h$. Since $|\alpha|=1$, projection onto $\alpha$ and the reverse triangle inequality yield
\begin{align}\label{eq:chords}
 |\alpha-q|+|q-x_0|
 &\geq 1-ux-vh+|q|-|x_0|\notag\\
 &=1+h\left(\sqrt{1+k^2}-uk-v\right)-|x_0|.
\end{align}
For completeness, the first term follows from
$|\alpha-q|\geq\Rea((\alpha-q)\overline\alpha)=1-ux-vh$.

Define $\phi(k)=\sqrt{1+k^2}-uk-v$. This is strictly convex and has its minimum at $k=u/v\in(-3,-1/10)$, by \eqref{eq:leftbox}. Hence, outside that interval,
\[
 \phi(k)\geq\min\{\phi(-3),\phi(-1/10)\}.
\]
Both endpoint values exceed $1/250$. To see this with rational comparisons, the function $u/10-\sqrt{1-u^2}$ is decreasing for $u\leq-1/5$ in the range at issue. Therefore
\begin{align*}
 \phi(-1/10)
 &\geq\frac{\sqrt{101}}{10}-\frac1{50}-\frac{\sqrt{24}}5\\
 &>\frac{251}{250}-\frac1{50}-\frac{49}{50}=\frac1{250}.
\end{align*}
Similarly, $-3\sqrt{1-v^2}-v$ is increasing for $v\geq2/5$, so
\begin{align*}
 \phi(-3)&\geq\sqrt{10}-3\sqrt{21/25}-\frac25\\
 &>\frac{1581}{500}-\frac{2751}{1000}-\frac25
   =\frac{11}{1000}>\frac1{250}.
\end{align*}
The four square-root comparisons follow by squaring positive rational numbers.

After $x_0$, the path crosses $\Im z=-h$. Complex conjugation gives the identical estimate for the lower portion. Applying the polygonal lower bound for variation to these ordered crossing points and using \eqref{eq:chords} gives
\begin{align*}
 \ell(\gamma)
 &>2r\left(1+\frac{h}{250}-|x_0|\right)\\
 &>2r\left(1+\frac{h}{250}-3\eta\right)
   =2r(1+\eta)>2+\eta,
\end{align*}
because $h/250=4\eta$. This contradicts the assumed length bound. No simplicity assumption on the path has been used: the selected upper crossing, real crossing, and lower crossing occur in that order.

\begin{proposition}\label{prop:variation}
Every continuous path joining distinct zeros of $P$ in $\{|P|\leq1\}$ has total variation greater than $2+\eta$.
\end{proposition}
\begin{proof}
Proposition~\ref{prop:pair} excludes all but the remaining conjugate pair. The preceding argument excludes every rectifiable path of length at most $2+\eta$ for that pair. Nonrectifiable paths have infinite variation.
\end{proof}

\section{From variation to Hausdorff length}\label{sec:hausdorff}
\subsection{The continuum argument}
A lower bound for the variation of a given parametrization does not by itself imply the same bound for the measure of its image. A bound that excludes \emph{every} rectifiable connector is different: it transfers by the general finite-length continuum fact below. We include its proof to make the argument self-contained and make no novelty claim for the lemma.

\begin{lemma}\label{lem:continuum}
Let $K\subset\C$ be compact and connected with $L=\HH(K)<\infty$. For any $a,b\in K$, there is a rectifiable path from $a$ to $b$ contained in $K$ with length at most $L$.
\end{lemma}
\begin{proof}
If $a=b$, the constant path suffices. For each positive integer $m$, the definition of Hausdorff measure gives a countable cover of $K$ by sets of diameter less than $1/m$ whose diameter sum is less than $L+1/m$. Enlarge them to open sets, with total increase of diameters less than $1/m$ and with each resulting diameter less than $2/m$. Compactness gives a finite subcover. Discard members not meeting $K$; denote the remaining sets by $U_1,\ldots,U_N$.

Connect indices $i,j$ when $U_i\cap U_j\cap K\ne\varnothing$. This finite graph is connected. Otherwise the unions over a partition into graph components would give two disjoint nonempty relatively open sets covering $K$, contradicting connectedness. Choose a simple chain from a set containing $a$ to a set containing $b$, and choose a point of $K$ in every successive intersection. Joining these points, together with $a,b$, by line segments gives a polygonal path of length at most the sum of the diameters of the sets in the chain, hence less than $L+2/m$.

Each line segment has endpoints in $K$ and length less than $2/m$, so the polygonal path lies in the $2/m$-neighborhood of $K$. Parametrize it at constant speed on $[0,1]$. These paths have a common Lipschitz bound $L+2$ and images in a common compact set. The Arzel\`a--Ascoli theorem provides a uniformly convergent subsequence. Its limit has endpoints $a,b$ and image in $K$. For every finite partition of $[0,1]$, the corresponding chord sum converges and is at most $L$ in the limit. Taking the supremum over partitions proves that the limiting path has length at most $L$.
\end{proof}

\begin{proof}[Proof of Theorem~\ref{thm:main}]
Section~\ref{sec:roots} proves the assertions about the polynomial and its zeros. Let $K=\gamma([0,1])$ for an allowed continuous path joining distinct zeros. If $\HH(K)\leq2+\eta$, then $K$ is a compact connected set of finite length. Lemma~\ref{lem:continuum} supplies a rectifiable path in $K$ joining the same endpoints with length at most $2+\eta$. Since $K\subset\{|P|\leq1\}$, this contradicts Proposition~\ref{prop:variation}. Thus \eqref{eq:main} holds. The total-variation assertion also follows directly from that proposition.
\end{proof}

\subsection{General consequences for earlier counterexamples}
The following proposition isolates the transfers mentioned in the introduction. They apply to any polynomial satisfying the hypotheses, independently of our particular coefficients or separator construction.

\begin{proposition}[Hausdorff and dilation transfers]\label{prop:transfers}
Let $f$ be monic of degree $n\geq2$, with all zeros in $\{|z|\leq\rho\}$ for some $0<\rho<1$. Let $L>0$. Suppose that every rectifiable path joining distinct zeros of $f$ in $U_f=\{|f|<1\}$ has length greater than $L$. Then:
\begin{enumerate}
\item Every continuous path joining distinct zeros of $f$ in $U_f$ has image of Hausdorff measure greater than $L$.
\item For $1<s<\rho^{-1}$, put $f_s(z)=s^nf(z/s)$. This polynomial is monic of degree $n$, all its zeros lie in the open unit disk, and every continuous path joining distinct zeros of $f_s$ in $\{|f_s|\leq1\}$ has image of Hausdorff measure greater than $sL$.
\end{enumerate}
If $f\in\Q[z]$ and $s\in\Q$, then $f_s\in\Q[z]$.
\end{proposition}
\begin{proof}
For the first assertion, suppose that a continuous connector has image $K$ with $\HH(K)\leq L$. This image is compact and connected. Lemma~\ref{lem:continuum} supplies a rectifiable connector in $K\subset U_f$ with length at most $L$, a contradiction.

For the second assertion, monicity follows from the factor $s^n$, and the zeros are the zeros of $f$ multiplied by $s$. Their moduli are at most $s\rho<1$. If $\gamma$ is an allowed connector for $f_s$, then $\widetilde\gamma=\gamma/s$ joins distinct zeros of $f$ and satisfies
\[
 |f(\widetilde\gamma(t))|=s^{-n}|f_s(\gamma(t))|
 \leq s^{-n}<1.
\]
The first assertion and the scaling law for Hausdorff measure give
\[
 \HH(\gamma([0,1]))
   =s\,\HH(\widetilde\gamma([0,1]))>sL.
\]
Finally, the coefficient of $z^k$ in $f_s$ is $s^{n-k}$ times that in $f$, proving the rationality assertion.
\end{proof}

\begin{remark}\label{rem:contribution}
Taking $L=2$ shows why neither a Hausdorff formulation nor the existence of some positive closed-sublevel margin, by itself, constitutes a new counterexample mechanism. The hypothesis with a strict lower bound is the form reported for ani's instance in~\cite{Cook,AniFormal}. If one assumes only that every rectifiable connector has length at least $2$, the same argument gives Hausdorff measure at least $2s>2$ after dilation. Thus even that weaker hypothesis suffices for a closed-sublevel obstruction above $2$.

These general transfers do not verify the specific polynomial \eqref{eq:P}, the five graphs of Table~\ref{tab:vertices}, or the margin $\eta=10^{-8}$ for that polynomial. Those are the explicit construction and certificate results established in Sections~\ref{sec:roots}--\ref{sec:length}. We make no claim that this margin is optimal, that the general transfers are new, or that coefficient differences alone establish an essentially new family.
\end{remark}

\section{Exact certificates and formal verification}\label{sec:certificates}
\subsection{Rational reproducibility}
The computational part consists only of finite algebraic identities and strict rational inequalities. The program in Appendix~\ref{app:checker} uses Python's arbitrary-precision \texttt{Fraction} arithmetic. It checks the polynomial scaling identity, the cubic factorization and root enclosures, all sixteen base segments and six endpoint rays, the complete root sign table, the two horizontal obstacles by conjugation, and the final numerical margins. It uses no floating-point arithmetic, numerical root solver, critical-point enumeration, or analytic continuation.

The supplementary file \texttt{graph\_fences.json} encodes precisely Table~\ref{tab:vertices}. Running
\begin{quote}\ttfamily\footnotesize
python3 verify\_graph\_separation.py graph\_fences.json graph\_certificate.json
\end{quote}
reconstructs the complete certificate. Each terminal interval has rational endpoints and a strictly negative rational upper Bernstein coefficient. The program also verifies that the terminal intervals cover the original parameter interval without gaps. The forty-four intervals counted in Table~\ref{tab:cells} refer only to the base graphs; the horizontal segment has a separate one-interval certificate. Assertions must be enabled; the program rejects execution with Python's \texttt{-O} option.

The accompanying archive also preserves an earlier critical-point certificate as supplementary provenance. It is not needed for the graph-separator proof in this paper. Both arithmetic checkers were rerun on 23 September 2026, and their generated JSON data agree with the stored certificates.

\subsection{Lean 4 verification and its precise scope}
The accompanying development uses Lean~4~\cite{Lean4} and mathlib~\cite{Mathlib4}. The three previously missing formal assertions have now been proved: the full arbitrary-image, closed-sublevel conclusion of Theorem~\ref{thm:main}, the compact-continuum connector lemma (Lemma~\ref{lem:continuum}), and the general transfers in Proposition~\ref{prop:transfers}. The verification run completed on 23 September 2026 comprises a project build, an axiom audit of all 314 named theorems, explicit manuscript and historical-statement adapters, and sequential kernel replay of all 28 project proof modules. Commands, exit codes, UTC timestamps, dependency revisions and verified source hashes are recorded in the accompanying supplement~\cite{VerificationSupplement}. The exact rational graph certificate was regenerated and agrees with the archived data.

The environment is Lean 4.33.1, compiler commit
\begin{center}\ttfamily\small
819816b2e0a3bf405af45ae5c7af2491d8f5bee6,
\end{center}
with mathlib commit
\begin{center}\ttfamily\small
0df444a360eaa60ab8c11dca51a86af692955474.
\end{center}
Every transitive dependency is pinned in \texttt{lake-manifest.json}. The audit reports only \texttt{propext}, \texttt{Classical.choice}, and \texttt{Quot.sound}; no admitted proof axiom occurs. Kernel replay checks each project module from its imported environment. It is not a replay of the entire mathlib dependency tree or an independent implementation of Lean's kernel.

The principal declarations are:
\begin{enumerate}
\item \texttt{JSP864.\allowbreak manuscript\_\allowbreak main\_\allowbreak theorem}: monicity, degree seven, simple roots, open-unit-disk root bounds, and the full inequality $\HH(\gamma([0,1]))>2+\eta$ for arbitrary continuous root-joining paths in the closed unit sublevel set. The path is not assumed injective or rectifiable.
\item {\small\texttt{JSP864.\allowbreak compact\_\allowbreak connected\_\allowbreak rectifiable\_\allowbreak connector}}: for a compact preconnected set of finite $\HH$ and any two points in it, a path in the set exists with total variation at most its $\HH$. No path-connectedness or local-connectedness hypothesis is added.
\item \texttt{JSP864.\allowbreak hausdorff\_\allowbreak dilation\_\allowbreak transfers}: both conclusions of Proposition~\ref{prop:transfers}, including monicity, degree and the unit-disk root bounds of the dilated polynomial. The coefficient identity and rationality are checked by \texttt{JSP864.\allowbreak dilation\_\allowbreak eval} and \texttt{JSP864.\allowbreak rational\_\allowbreak dilation\_\allowbreak map}. The dilation is represented by mathlib's \texttt{Polynomial.scaleRoots}.
\end{enumerate}
The formal connector proof first extends the capped intrinsic-distance argument to any nonnegative length threshold. It then constructs short paths in shrinking open neighborhoods of the compact set, reparametrizes them with controlled Lipschitz constants, and uses compactness of a nested family of candidates in the product topology. The limiting connector lies in the original set and has the required length bound. This proves the same lemma as the polygonal and Arzel\`a--Ascoli argument in Section~\ref{sec:hausdorff}; it does not assume that lemma as an extra axiom.

The file \texttt{ManuscriptCheck.lean} explicitly checks the manuscript statements, including the positive real threshold in Proposition~\ref{prop:transfers}. The older open-sublevel declarations \texttt{JSP864.\allowbreak original\_\allowbreak counterexample} and \texttt{JSP864.\allowbreak not\_\allowbreak erdos\_\allowbreak 1041} remain available. These concern our polynomial and are distinct from Cook's formalization of ani's example~\cite{AniFormal}. The separate historical-statement adapter reproduces the root-multiset and open-sublevel formulation in the Formal Conjectures catalogue~\cite{FormalConjectures}, which is not a prize adjudication authority.

The verification claims concern the listed mathematical statements and their dependencies. They do not mean that every sentence, literature comparison or priority statement in the article is a Lean theorem. Earlier records of partial formalization are preserved separately and are superseded, for these three assertions, by the current sources and successful audit.

\subsection{Availability and preparation}
The Lean development and its verification evidence are archived at
\begin{quote}\small
\url{https://github.com/zx967/JSP-000864-Lean4}.
\end{quote}
The fixed source snapshot cited in this article is commit
\begin{center}\ttfamily\small
e8323addf7e11cb3255bee8af6bd1f9e08e552a3.
\end{center}
Reference~\cite{VerificationSupplement} links to that snapshot, so later changes to the default branch do not change the cited proof sources. The repository is private at the time of preparation. The authors plan to make it public after this article is announced on arXiv; until then, repository access is restricted to authorized users. Thus the repository links identify the archive but are not yet publicly accessible. No arXiv identifier has yet been assigned in this manuscript.

The stored verification report records the authors' completed run; it is not a claim that a GitHub Actions run has succeeded or that an independent third party has reproduced the verification. The repository includes a workflow and the command \texttt{python3 scripts/verify\_release.py} for repeating the checks with the pinned dependencies.

The submission package contains the manuscript source, this PDF, both arithmetic checkers and their input/output data, and the Lean source package with its build instructions, fresh verification logs, and a separate copy of the earlier archived records. No external web service is required to run the rational checker. The appendix and the tables specify the finite computations needed for the proof.

Generative AI (OpenAI ChatGPT/Codex) assisted in the mathematical development, code preparation, verification workflow, and writing of this manuscript. No independent external referee review is represented by the computational or formal verification records.

\appendix
\section{Complete rational checker}\label{app:checker}
The following program is reproduced without its introductory documentation string; its executable content is unchanged. The JSON input is the exact rational encoding of Table~\ref{tab:vertices} included in the supplement. Internally, the names \texttt{root\_1}, \texttt{root\_upper\_right}, and \texttt{root\_upper\_left} denote $A_0,B_+,C_+$, respectively. The list order of the root boxes is $1,\alpha_1,\overline{\alpha_1},\alpha_2,\overline{\alpha_2},\alpha_3,\overline{\alpha_3}$.
\begin{lstlisting}
from fractions import Fraction as F
from math import comb
from pathlib import Path
import json,sys

if not __debug__:
    raise SystemExit('Run without -O.')
e=F(1,1000);a=F(11,4);b=-F(55,16);delta=F(1,10**30);r=1-delta
g=[F(0),F(1),b,a,-a*e,-b*e**3,-e**5,F(1)]
Q=[F(-1),e**6,b*e**5,a*e**4,-a*e**4,-b*e**5,-e**6,F(1)]
kappa=F(1,2000)

def rational(x):
    assert isinstance(x,(int,str)) and not isinstance(x,bool),'Use integer or rational-string coordinates.'
    return F(x)

# Real polynomial arithmetic, with coefficients in increasing degree.
def add(p,q):
    z=[F(0)]*max(len(p),len(q))
    for i,x in enumerate(p):z[i]+=x
    for i,x in enumerate(q):z[i]+=x
    return z
def mul(p,q):
    z=[F(0)]*(len(p)+len(q)-1)
    for i,x in enumerate(p):
        for j,y in enumerate(q):z[i+j]+=x*y
    return z
def power(p,n):
    z=[F(1)]
    for _ in range(n):z=mul(z,p)
    return z
def evaluate(p,t):
    ans=F(0)
    for x in p[::-1]:ans=ans*t+x
    return ans
def real_on_line(p,q):
    x=[p[0],q[0]-p[0]];y=[p[1],q[1]-p[1]]
    ans=[F(0)]*8
    for n,c in enumerate(g):
        for j in range(0,n+1,2):
            term=mul(power(x,n-j),power(y,j))
            scale=c*comb(n,j)*(-1)**(j//2)
            ans=add(ans,[scale*t for t in term])
    return ans
def bernstein(p):
    n=len(p)-1
    return [sum(p[k]*F(comb(j,k),comb(n,k)) for k in range(j+1)) for j in range(n+1)]
def split(B):
    left=[B[0]];right=[B[-1]];row=B
    while len(row)>1:
        row=[(x+y)/2 for x,y in zip(row,row[1:])]
        left.append(row[0]);right.append(row[-1])
    return left,list(reversed(right))
def certify_segment(p,q):
    coeff=real_on_line(p,q);coeff[0]+=kappa
    stack=[(bernstein(coeff),0,F(0),F(1))];leaves=0;max_depth=0;cells=[]
    while stack:
        row,d,lo,hi=stack.pop()
        if max(row)<0:
            leaves+=1;max_depth=max(max_depth,d)
            cells.append({'interval':[str(lo),str(hi)],'bernstein_upper':str(max(row))})
        else:
            assert d<24,'Segment certificate failed.'
            l,rr=split(row);mid=(lo+hi)/2
            stack.extend([(l,d+1,lo,mid),(rr,d+1,mid,hi)])
    cells.sort(key=lambda x:F(x['interval'][0]))
    assert cells[0]['interval'][0]=='0' and cells[-1]['interval'][1]=='1'
    assert all(x['interval'][1]==y['interval'][0] for x,y in zip(cells,cells[1:]))
    return {'dyadic_leaves':leaves,'maximum_depth':max_depth,'cells':cells}
def certify_ray(d):
    coeff=real_on_line((F(0),F(0)),d)
    upper=coeff[7]+sum(max(x,F(0)) for x in coeff[:7])
    assert upper<-kappa
    return str(upper)
def linear_interval(A,box,constant=F(0)):
    lo=hi=constant
    for c,(l,h) in zip(A,box):
        if c>=0:lo+=c*l;hi+=c*h
        else:lo+=c*h;hi+=c*l
    return lo,hi
def graph_sign(fence,box):
    T=fence['T'];S=fence['S'];points=fence['points']
    ts=[sum(c*x for c,x in zip(T,p)) for p in points]
    ss=[sum(c*x for c,x in zip(S,p)) for p in points]
    tlo,thi=linear_interval(T,box)
    if thi<=ts[0]:m=ss[0]/ts[0];d=F(0)
    elif tlo>=ts[-1]:m=ss[-1]/ts[-1];d=F(0)
    else:
        found=False
        for j in range(len(ts)-1):
            if ts[j]<=tlo<=thi<=ts[j+1]:
                 m=(ss[j+1]-ss[j])/(ts[j+1]-ts[j]);d=ss[j]-m*ts[j];found=True;break
        assert found,'Root box crosses a graph knot.'
    lo,hi=linear_interval([S[j]-m*T[j] for j in range(2)],box,-d)
    assert lo>0 or hi<0
    return 1 if lo>0 else -1

def verify(data):
    assert all(Q[j]==g[j]*e**(7-j) for j in range(1,8))
    assert r**7*(1+e**7*kappa)>1
    A=1-e**6;B=A-b*e**5;C=B-a*e**4
    cubic=[C-2*A,2*(B-3),4*A,F(8)]
    pairs=[(F(3,5),F(13,20),F(3,4),F(4,5)),
           (-F(1,4),-F(1,5),F(24,25),F(49,50)),
           (-F(91,100),-F(9,10),F(41,100),F(11,25))]
    for lo,hi,vl,vh in pairs:
        assert evaluate(cubic,lo)*evaluate(cubic,hi)<0
        largest=max(abs(lo),abs(hi));smallest=min(abs(lo),abs(hi))
        assert largest**2+vl**2<1 and smallest**2+vh**2>=1
    S=[F(1),A,B,C,B,A,F(1)]
    assert [-S[0]]+[S[j-1]-S[j] for j in range(1,7)]+[S[-1]]==Q
    boxes=[[(F(1)/e,F(1)/e),(F(0),F(0))]]
    labels=['1']
    for j,(lo,hi,vl,vh) in enumerate(pairs,1):
        boxes.append([(lo/e,hi/e),(vl/e,vh/e)]);labels.append('upper_'+str(j))
        boxes.append([(lo/e,hi/e),(-vh/e,-vl/e)]);labels.append('lower_'+str(j))
    fences=[];reports=[]
    for raw in data:
        f={'name':raw['name'],'T':[rational(x) for x in raw['T']],'S':[rational(x) for x in raw['S']],
           'points':[tuple(rational(x) for x in p) for p in raw['points']]}
        T,S,points=f['T'],f['S'],f['points']
        assert T[0]*S[1]-T[1]*S[0]!=0
        ts=[sum(c*x for c,x in zip(T,p)) for p in points]
        assert ts[0]<0<ts[-1] and all(x<y for x,y in zip(ts,ts[1:]))
        seg=[certify_segment(p,q) for p,q in zip(points,points[1:])]
        rays=[certify_ray(points[0]),certify_ray(points[-1])]
        reports.append({'name':f['name'],'segments':seg,'ray_upper_bounds':rays})
        fences.append(f)
        if f['name']!='root_1':
            fences.append({'name':f['name'].replace('upper','lower'),
                            'T':[T[0],-T[1]],'S':[S[0],-S[1]],
                            'points':[(x,-y) for x,y in points]})
    table=[[graph_sign(f,box) for f in fences] for box in boxes]
    equal=[(i,j) for i in range(7) for j in range(i+1,7) if table[i]==table[j]]
    assert equal==[(3,4)],'Only the middle-left conjugate pair may remain connected.'
    # Entire short horizontal obstacles, independently certified in Bernstein form.
    upper_short=certify_segment((-F(3,100),F(1,100)),(-F(1,1000),F(1,100)))
    # Geometry constants from the all-path argument.
    eta=F(1,10**8);h=e/100;cross=3*eta;budget=(1+eta/2)/r
    assert r**(-7)-1<8*delta
    assert F(5,2)*(budget**2-1)<cross and 8*delta/e**6<cross
    assert F(251,250)**2<F(101,100) and F(49,50)**2>F(24,25)
    assert F(251,250)-F(1,50)-F(49,50)==F(1,250)
    assert F(1581,500)**2<10 and F(917,1000)**2>F(21,25)
    assert F(1581,500)-3*F(917,1000)-F(2,5)>F(1,250)
    assert 2*r*(1+h/250-cross)>2+eta
    return {'proof_method':'continuous piecewise-linear graph separators; rational Bernstein certificates',
            'fences':[f['name'] for f in fences],'root_labels':labels,'root_signs':table,
            'unseparated_pair':[labels[i] for i in equal[0]],'segment_certificates':reports,
            'short_horizontal_certificate':upper_short,'length_threshold':str(2+eta)}

if __name__=='__main__':
    data=json.loads(Path(sys.argv[1]).read_text())
    result=verify(data)
    print(json.dumps(result,indent=2))
    if len(sys.argv)>2:Path(sys.argv[2]).write_text(json.dumps(result,indent=2))
\end{lstlisting}

\begingroup\raggedright

\endgroup
\end{document}